\documentclass[10pt]{amsart}

\usepackage{amsmath, amsfonts}
\usepackage[all]{xy}
\usepackage{footnote}

\title{Nim on the Complete Graph}
\author{Lindsay Erickson}
\begin{document}

\long\def\symbolfootnote[#1]#2{\begingroup%
\def\thefootnote{\fnsymbol{footnote}}\footnote[#1]{#2}\endgroup}
\theoremstyle{plain}
\newtheorem{Theorem}{{Theorem}}[section]
\newtheorem{Lemma}[Theorem]{Lemma}
\newtheorem{Proposition}[Theorem]{Proposition}
\newtheorem{Corollary}[Theorem]{Corollary}
\newtheorem{Definition}[Theorem]{Definition}
\newtheorem{Example}[Theorem]{Example}
\newtheorem{Figure}[Theorem]{Figure}
\newtheorem{Conjecture}[Theorem]{Conjecture}

\maketitle

\symbolfootnote[0]{Accepted for publication, \textit{Ars Combinatoria,} November 2009}

\section*{Abstract}
The game of Nim as played on graphs was introduced  in \cite{MR1992342} and extended in \cite{MR1992343} by Masahiko Fukuyama.  His papers detail the calculation of Grundy numbers for graphs under specific circumstances.  We extend these results and introduce the strategy for even cycles.  This paper examines a more general class of graphs by restricting the edge weight to one.  We provide structural conditions for which there exist a winning strategy.  This yields the solution for the complete graph.  

\begin{section}{Background}
The general nontrivial game of Nim is a two-person combinatorial game \cite{MR1808891} consisting of at least three piles of stones where players alternate turns, selecting first a pile from which stones will be removed, and then a strictly positive number of stones to remove.  The game terminates when there are no more stones, and the winner is the player who takes the last stone or stones.  Players must always remove at least one stone, and can only remove stones from a single pile during their turn.

The solution to this general game of Nim is well known and can be found in \cite{MR1808891}.  More interestingly, the solution for Nim can be applied to other two-player combinatorial games.  Masahiko Fukuyama extended Nim to finite graphs by first fixing an undirected graph, assigning to each edge a positive integer, and placing a game position indicator piece at some vertex.  From this indicator piece, the game begins and proceeds with alternate moves from two players according to the following rules.  First, a player chooses an edge incident with the piece.  The player decreases the value of this edge to any non-negative integer and moves the indicator piece to the adjacent vertex along this edge.  Game play ends when a player is unable to move since the value of each edge incident with the piece equals zero.  The player unable to move is the loser of the game \cite{MR1992342}.  Nim on graphs differs from the general game in that it might not be the case that all weight is removed from the graph.  We assume that a player with a winning strategy would choose to use it.

Very few results are known for Nim on graphs, adding to its appeal.  Also, the strategies employed for ordinary Nim are not applicable to Nim on graphs.  We will note some fundamental results of Nim on graphs following the definitions.  In Section 2, we improve upon Fukuyama's result by showing the unique winning strategy for even cycles. Next in Section 3 we define a structure and strategy that leads to a first player victory.  Then in Section 4 we show that the presence of this structure yields a first player win. This leads to the solution of the complete graph when each edge is given a weight of one.  

\begin{subsection}*{Definitions}

The graphs we will consider are finite and undirected with no multiple edges or loops.  We will often want to label the vertices and edges.  When we do, the edge between vertex $v_i$ and $v_j$ will be denoted $e_{ij}$.  Graph theory terminology, including path, vertex degree, and graph isomorphism, will be assumed as found in \cite{MR2107429}.

\begin{Definition}
Given a graph $G$ with edge set $E(G)$ and vertex set $V(G)$ we will call the non-negative integer value assigned to each $e \in E(G)$ the \textbf{weight} of the edge and denote the weight of edge $e_{ij}$ by $\omega(e_{ij})$. 
\end{Definition}

For any graph $G$ we assume  $\omega(e_{ij}) \neq 0$ for all $e_{ij} \in E(G)$ at the start of a game.  When an edge is decreased to $\omega(e)=0$ we will delete it from the graph entirely, since it is no longer a playable edge.  Given a game graph $G$ with weight assignment $\omega_G (e)$, denote by $P_1$ the first player to move from the starting vertex, and denote by $P_2$ the player to move after $P_1$.  The indicator piece $\Delta$ denotes the vertex from which a player is to move.  We will always enumerate vertices in such a way that $\Delta$ is on $v_1$ at the start of a game.

\begin{Definition}
For either player and from a given position $\Delta$ on vertex $v_j$, we define the set of vertices to which a player may legally move to from $\Delta$ the \textbf{options} of the player.  The options of player $i$ at vertex $v_j$ will be denoted by $O(P_i, v_j)$.
\end{Definition}

Certainly for a vertex to exist in the set of options the incident edge must be adjacent to $\Delta$.  Thus $O(P_i, v_j)=\{v_k \in V(G) : \Delta = v_j; \ e_{jk} \in E(G);  \ \omega(e_{jk}) \neq 0\}$.  We will omit $v_j$ when the position of $\Delta$ is apparent.

\begin{Definition}
We will say that a pair of $P_i$'s options are \textbf{isomorphic} if given two options, $v_j, v_k \in O(P_i, v_i)$, there exists a graph isomorphism between $v_j$ and its neighbors and $v_k$ and its neighbors.  We will say that two options are \textbf{identical} if in addition to being isomorphic, the options also have the same weight assignment.
\end{Definition}

Notice that the definition of isomorphic requires that the vertices of isomorphic options have the same degree, and that there is a bijection between the options of the vertices in the set of isomorphic options (see Example 1.4).  

\vspace{.25in}

\begin{Example}

The options at $\Delta$ are isomorphic but not identical.  

\end{Example}

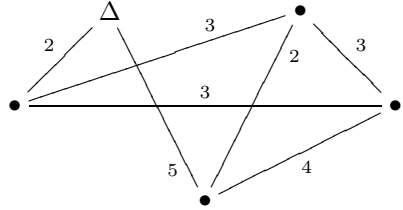
\begin{center}
$
\xygraph{ 
!{(-1,2) }*+{\Delta}="a" 
!{(1,2) }*+{\bullet}="b" 
!{(2,1) }*+{\bullet}="c" 
!{(0,0)}*+{\bullet}="d"
!{(-2,1)}*+{\bullet}="e" 
"a"-"e"_{2} 
"a"-"d"_(.8){5}
"b"-"c"^{3}
"b"-"d"^(.2){2}
"b"-"e"_(.3){3}
"c"-"d"^{4}
"c"-"e"_{3}
} 
$
\end{center}

When we refer to a player winning a graph we precisely mean that a player can win a game played on that graph under the specified weight assignment and starting at a specified vertex.  If no starting vertex is specified, it is true that a player can win that graph starting at any vertex.  When we say that a player is on an odd path or that a player has an odd path option, we mean that there is an odd path from $\Delta$ to a vertex of degree one.  However, when we say even path, we precisely mean that all options from $\Delta$ to a vertex of degree one are even paths.   Notice that if $G$ is an odd path itself, there is an odd path option at any vertex, hence there is no loss of generality by not specifying the position of $\Delta$.  Such is not the case in $G$ is an even path, since it is possible to position $\Delta$ on vertices of $G$ in which both options are indeed odd paths.

\begin{Definition}
From a particular position, if the first player to move can win for any of the second player's moves, we call this position a {\bf{p-position}}.  If the second player to move from this position can win for any of the first player's moves, we call this a \bf{0-position} \cite{MR1992342}.
\end{Definition}

The terms $p$-position and $0$-position come from the \textit{positive} and \textit{zero} Grundy number of that particular position \cite{MR1808891}.  Grundy numbers are used heavily in many areas of two-person combinatorial game theory.  Two properties especially important to keep in mind are that when a player is on a $0$-position all moves are to $p$-positions, and that when a player is on a $p$-position there is always at least one move to a $0$-position.  Essentially, this means that a player with an advantage at the beginning can keep it with a winning strategy.  Any position on an odd path is a $p$-position for $P_1$, as is any position on an odd cycle.  Starting at either vertex of degree one on an even path is a $0$-position for $P_1$, and starting at any vertex that leaves two even paths from $\Delta$ is also a $0$-position for $P_1$.  However, as noted above, if $\Delta$ started on a vertex that leaves two odd paths on this even path, the position is a $0$-position for the first player.

The Grundy number of a position in ordinary Nim not only told which player has an advantage at any given position, it also told that player what move to make when the Grundy number was positive.  This is not the case with Nim on graphs.  Knowing that you can win with Grundy number calculations does not tell you what strategies should be employed to defeat your opponent.  The calculations of the Grundy numbers for trees, paths, cycles, and certain bipartite graphs can be found in \cite{MR1992342, MR1992343}.

\end{subsection}

\begin{subsection}*{Paths and Cycles}

The first player will win an odd path from any starting position since any vertex on an odd path has an odd path option.  The strategy for the first player is to remove all weight from the edge on the odd path option.  This leaves $P_2$ on an even path at a vertex of degree one, which as mentioned previously, is a $0$-position for $P_2$.  On the other hand, the second player will win when starting from a vertex in which every option is even path to a vertex of degree one  (Example 1.6).
  
\begin{Example}
Starting out at $v_1$ we can see that the first player will move along the player's only choice of edges, once again, removing the entire edge. \\
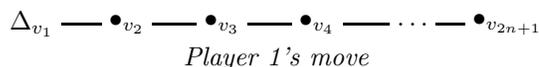
\begin{center}
$
\xygraph{ 
!{(0,0) }*+{\Delta_{v_1}}="a" 
!{(1,0) }*+{\bullet_{v_2}}="b" 
!{(2,0) }*+{\bullet_{v_3}}="c" 
!{(3,0)}*+{\bullet_{v_4}}="d"
!{(4,0)}*+{\ldots}="e" 
!{(5,0)}*+{\bullet_{v_{2n+1}}}="f"
"a"-"b" 
"b"-"c"  
"c"-"d" 
"d"-"e"
"e"-"f"
} 
$

$\mbox{Player 1's move}$

\end{center}

If the first player did not remove the entire edge, it would be a faster victory for the second player.

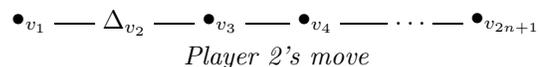
\begin{center}
$
\xygraph{ 
!{(0,0) }*+{\bullet_{v_1}}="a" 
!{(1,0) }*+{\Delta_{v_2}}="b" 
!{(2,0) }*+{\bullet_{v_3}}="c" 
!{(3,0)}*+{\bullet_{v_4}}="d"
!{(4,0)}*+{\ldots}="e" 
!{(5,0)}*+{\bullet_{v_{2n+1}}}="f"
"a"-"b" 
"b"-"c"  
"c"-"d" 
"d"-"e"
"e"-"f"
}
$

$\mbox{Player 2's move}$

\end{center}
The second player is now on an odd path either way the second player decides to move.
\end{Example}

Moving from paths to cycles we see that the first player to start from an odd cycle can always win.  When considering an odd cycle, $P_1$ will move along either edge, removing all of the weight.  This will leave $P_2$ on an even path.  Since the second player to move from an even path can always win, the first player to move from an odd cycle can always win.

The strategy for either player on an odd cycle contrasts greatly with the strategy for the even cycle.   In \cite{MR1992343}, Fukuyama calculated the Grundy number of even cycles.  However, the calculation does not lend itself to a winning strategy for the player with the advantage. If the first player to move on an even cycle removed all weight on an edge, the first player has left the second player on an odd path.  In fact, both players would want to avoid ``breaking" the even cycle.  Because of this, the player who is able to avoid breaking the even cycle will win.  We explain the strategy for even cycles in Section 2.

\vspace{1in}

\begin{Example}
In the $C_4$ on the left, the first player to move has the advantage.  In the $C_4$ on the right, the second player has the advantage.  
\end{Example}

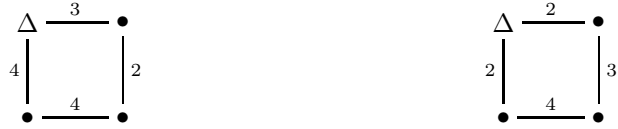
\begin{center}
$
\xygraph{ 
!{(0,0) }*+{\bullet}="a" 
!{(1,1) }*+{\bullet}="b" 
!{(0,1) }*+{\Delta}="c" 
!{(1,0)}*+{\bullet}="d" 
!{(5,0) }*+{\bullet}="e" 
!{(6,1) }*+{\bullet}="f" 
!{(5,1) }*+{\Delta}="g" 
!{(6,0)}*+{\bullet}="h" 
"c"-"b" ^{3}
"a"-"d" ^{4} 
"a"-"c" ^{4}
"b"-"d" ^{2}
"e"-"h" ^{4} 
"f"-"h" ^{3}
"e"-"g" ^{2}
"f"-"g" _{2}
} 
$
\end{center}

The important part of this result is that the weights of the edges matter for even cycles.  It is not difficult to show that if the weight on each edge equals one, the second player to start on an even cycle will win.  This is used extensively in Section 4. 
\end{subsection}

\end{section}

\begin{section}{Strategy for Even Cycles}
The Grundy number is 0 when $\omega(e) = k$ for all edges in an even cycle and for any $k \geq 0$, hence the second player to start has the advantage \cite{MR1992343}.  Consider first the strategy for $P_2$ when $\omega(e) = 2$ for all edges on an arbitrary even cycle.

From any starting position and for either edge, $P_1$ only has the choice of reducing that edge to a weight of 1 or 0.  Notice that $P_1$ would not want to make an odd path for $P_2$ by reducing to 0.  Thus assume without loss of generality that $P_1$ moves to $v_2$ and reduces $e_{12}$ to $\omega(e_{12}) = 1$.  Then $P_2$'s next $0$-position option is to move to $v_3$ leaving $\omega(e_{23}) = 1$, since moving back to $v_1$ requires that $P_2$ create an odd path for $P_1$.  Continuing on in this way $P_1$ and $P_2$ will move to $v_{2j}$ and $v_{2j+1}, (1 \leq j \leq n-1)$ respectively until $P_1$ is back at $v_1$ and only an even cycle with $\omega(e) = 1$ for all $e \in E(C_{2n})$, which as mentioned is a $P_2$ victory. 

In the above case, $P_1$ was immediately forced to reduce the weight of an edge beyond the minimum weight of any edge.  Now assume that the weights of the edges on an even cycle are arbitrary.  It will still be the case that neither player wants to break the even cycle, and that the first player forced to decrease a weight below the minimum will lose.  This means we can look at even cycles with arbitrary weighting assignments in the following way:

\begin{Proposition}
Assume $G = C_{2n}$ and that $\omega_G$ is some arbitrary weight assignment for $G$.  Assume $\min_{e\in E(G)}(\omega_G (e)) = m$.  Let $G^{'}$ be the graph formed from $G$ under $\omega_{G^{'}} (e) = \omega_G (e) - m$ with the same starting vertex.  Then the $p$-positions of $G$ are the $p$-positions of $G^{'}$ with the winning strategy for $P_1$ or $P_2$ on $G$ following from that on $G^{'}$.
\end{Proposition}

\begin{proof}
Note that $G^{'}$ is no longer an even cycle since at least one edge and perhaps all edges of $G$ are deleted under $\omega_{G^{'}}(e)$.  By Proposition 6.2 in \cite{MR1992343} which gives a calculation of the Grundy number of even cycles, the Grundy number of $G$ is determined in part by the Grundy number of $G^{'}.$  As the Grundy number of an odd path is positive and an even path is zero, the first player wins $G$ if there is at least one odd path starting from $\Delta$ in $G^{'},$ and the second player wins $G$ if all paths starting from $\Delta$ in $G^{'}$ are even.

To see that the strategy for playing $G$ follows from that for $G^{'},$ first consider a graph with a positive Grundy number.  On an odd path, we know that $P_1$ removes all weight on the incident edge.  Since the Grundy number of $G$ is positive, so is the Grundy number of $G^{'}$.  Hence $G^{'}$ contains an odd path.  The previous paragraph implies that $P_1$ will move in the direction of the odd path in $G^{'}$ decreasing the weight of $e_{12}$ to zero.  In $G$, this corresponds to a move from $v_1$ to $v_2$ and a decrease of $\omega_G (e_{12})$ by $\omega_{G^{'}} (e_{12})$ to $m$ since $\omega_{G^{'}}(e) = \omega_{G}(e) - m$ for all $e \in E(G)$.  

First assume that $P_2$ moves back to $v_1$.  Following this move, $\omega_{G} (e_{12}) = m^{'} < m$ and we can now compare the strategy for $P_1$ to the strategy for some graph $G^{''}$ formed from $G$ with $\omega_{G^{''}} (e) = \omega_{G} (e) - m^{'}$ where $G$ has been played two moves (see Example 2.2).  Since $G^{''}$ is an odd path of length $2n-1$ the first player has a winning strategy.

\begin{Example}  Below are graphs of $G$ and $G^{'}$ at the start of a game on even cycles.  In this game $m = 2$, and since an odd path exists in $G^{'}$ we have a winning strategy for $P_1$.

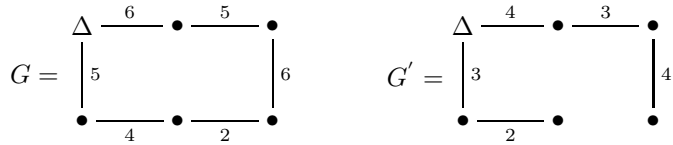
\begin{center}
$\xygraph{
!{(-.5,.5)}*+{G=}
!{(3.5,.5)}*+{G^{'}=}
!{(0,1) }*+{\Delta}="a" 
!{(1,1) }*+{\bullet}="b" 
!{(2,1) }*+{\bullet}="c" 
!{(2,0) }*+{\bullet}="d" 
!{(1,0) }*+{\bullet}="e" 
!{(0,0) }*+{\bullet}="f" 
!{(4,1) }*+{\Delta}="g" 
!{(5,1) }*+{\bullet}="h" 
!{(6,1) }*+{\bullet}="i" 
!{(6,0) }*+{\bullet}="j" 
!{(5,0) }*+{\bullet}="k" 
!{(4,0) }*+{\bullet}="l" 
"a"-"b"^{6}
"b"-"c"^{5}
"c"-"d"^{6}
"d"-"e"^{2}
"e"-"f"^{4}
"a"-"f"^{5}
"g"-"h"^{4}
"h"-"i"^{3}
"i"-"j"^{4}
"k"-"l"^{2}
"g"-"l"^{3}
}
$
\end{center}

In the case that $P_2$ goes back to $v_1$ lowering the weight of the edge beyond two, we have the following graphs $G$ and $G^{''}$.  Notice that in $G^{''}$  there is an odd path of length $2n-1$ since $\min_{e \in E(G)}(\omega(e)) = 1$ after the first two moves.

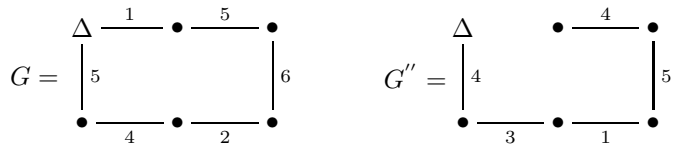
\begin{center}
$\xygraph{
!{(-.5,.5)}*+{G=}
!{(3.5,.5)}*+{G^{''}=}
!{(0,1) }*+{\Delta}="a" 
!{(1,1) }*+{\bullet}="b" 
!{(2,1) }*+{\bullet}="c" 
!{(2,0) }*+{\bullet}="d" 
!{(1,0) }*+{\bullet}="e" 
!{(0,0) }*+{\bullet}="f" 
!{(4,1) }*+{\Delta}="g" 
!{(5,1) }*+{\bullet}="h" 
!{(6,1) }*+{\bullet}="i" 
!{(6,0) }*+{\bullet}="j" 
!{(5,0) }*+{\bullet}="k" 
!{(4,0) }*+{\bullet}="l" 
"a"-"b"^{1}
"b"-"c"^{5}
"c"-"d"^{6}
"d"-"e"^{2}
"e"-"f"^{4}
"a"-"f"^{5}
"h"-"i"^{4}
"i"-"j"^{5}
"j"-"k"^{1}
"k"-"l"^{3}
"g"-"l"^{4}
}
$
\end{center}

\end{Example}

Now assume that $P_2$ moves to $v_3$ and sets $\omega(e_{23})$ to $k$.  If $k > m$ then we know from $G^{'}$ that $P1$ moves back to $v_2$ setting $\omega(e_{23}) = m$.  Since $P_2$ is on an even path in $G^{'}$, the first player will win. If $k = m$, then $P_1$ still has an odd path in $G^{'}$ and thus will win $G$. Finally if $k < m$ then $k$ is the new minimum weight and there exists $G^{''}$ with $\omega_{G^{''}} (e) = \omega_{G} - k$ that is an odd path of length $2n-1$ for $P_1$.  In any case, the strategy for $P_1$ follows that for a graph with the lowest weight removed from every edge.  

When the Grundy number of $G^{'}$ is zero, $P_2$ mimics the strategy of $P_1$ above.

To establish the uniqueness of this strategy, we must show that any move except one to reduce the edge weight to $m$ on the odd path option results in a loss for the player who began on a $p$-position.  In fact, the strategy holds at every stage of game play.

Assume $P_1$ begins the game on a $p$-position and let $m$ be the minimum weight of any edge of $G = C_{2n}$ and $\Delta = v_1$ as before.  There exists an odd path option in $G^{'}$, the graph formed from $G$ under $\omega_{G^{'}}(e) = \omega_{G}(e) - m$ for all $e \in E(G)$.  Since taking an even path option results in a loss for $P_1$ by the above arguments, we assume that $P_1$ takes an odd path option.

Suppose that $P_1$ does not reduce the weight of $e_{12}$ to $m$.  We consider first the case when $\omega(e_{12}) = m^{'}$ for $0 \leq m^{'} < m$ following $P_1$'s move.  With $\Delta = v_2$ and $P_2$'s turn, we can look at a graph $G^{''}$ formed from $G$ under $\omega_{G^{''}}(e) = \omega_{G}(e) - m^{'}$.  Since $m^{'} < m$ we have that $G^{''}$ is a path of length $2n-1$.  Now $P_2$ may move along $G$ in the direction of the odd path in $G^{''}$ reducing the edges to $m^{'}$ as play progresses for the win.  Thus $P_1$ reducing any edge below $m$ results in a loss of advantage and a $P_2$ win.

Now suppose that $P_1$ reduces $e_{12}$ to $m^{''}$ for $m < m^{''}$ if possible, and that only one odd path option exists.  If $\omega(e_{12}) = m + 1$ then we have nothing to show at this step, and if there are two odd path options, we will simply repeat this following argument a second time.  With $\Delta = v_2$ and $P_2$'s turn we will let $P_2$ move back to $v_1$ reducing the weight of $e_{12}$ from $m^{''}$ to $m$.  In doing this, $P_2$ has left $P_1$ on an even or trivial path in the graph $G^{'}$ formed under the weight assignment $\omega_{G^{'}}(e) = \omega_{G}(e) - m$ after two moves on $G$.  Since this is a $0$-position for $P_1$, we have that $P_2$ now holds the winning strategy.  Thus using any other strategy on even cycles shifts the advantage to the player who originally started in a $0$-position.

\end{proof}
Once the minimum weight is removed from each of the edges, it becomes clear that $P_1$ will win if there is an odd path option from the starting vertex.  In the same way we know that $P_2$ will win if all first player options from the starting vertex are even paths in $G^{'}$ (Example 1.7). 

\end{section}

\begin{section}{A Structure Theorem}
\begin{Theorem}
Let $G = K_{2,j}$ for $j \geq 1$ and $\omega(e) = 1$ for each $e \in K_{2,j}$.  Assume that $\Delta$ is on a vertex in the partite set of size 2.  Then $P_2$ will always win the $K_{2,j}$.
\end{Theorem}

\begin{proof}
We proceed by induction on $j$.  Enumerate the vertices in the following way:  Let $\Delta = v_1$  and $v_2$ be the other vertex in the partite set of size 2.  Enumerate the vertices in the partition of size $j$ by $v_3, v_4, \ldots, v_{j+2}$.  

For $j=1$ we have an even path.  By previous work, this is a win for $P_2$.  Similarly, for $j = 2$ we have an even cycle in which each edge has $\omega(e)=1$ which we have also seen to be a win for $P_2$.  Now assume that this is true for all complete $K_{2,i}$ for $i \leq j$.  Consider the $K_{2,j+1}$ with $\Delta$ on $v_1$ in the partition of size 2.  Notice that all of $P_1$'s moves are identical since $O(P_1, \Delta = v_1) = \{v_3, v_4, \ldots, v_{j+3}\}$, all incident edges have weight 1, and $d(v_i) = 2$ for $3 \leq i \leq j+3$.

Without loss of generality, assume that $P_1$ moves to $v_3$.  Since $e_{13}$ is now gone, as $\omega(e_{13}) = 1$ at the start, $P_2$ only has one move, namely to $v_2$.  Now with $\Delta$ on $v_2$ and both players unable to move to $v_3$, we have $P_1$ on a $K_{2,j}$ (Figure 3.2).

\begin{Figure}

\begin{center}
$
\xygraph{
!{(2.5,2) }*+{\bullet^{v_1}}="a"
!{(5.88,2) }*+{\Delta^{v_2}}="b"
!{(8.4,0) }*+{\bullet_{v_3}}="c"
!{(7.56,-.5) }*+{\bullet_{v_4}}="d"
!{(6.72,-.7) }*+{\bullet_{v_5}}="e"
!{(5.88,-.85) }*+{\bullet_{v_6}}="f"
!{(5,-1) }*+{\bullet_{v_7}}="g"
!{(4.2,-1) }*+{\bullet_{v_8}}="h"
!{(3.36,-.85) }*+{\bullet_{v_9}}="i"
!{(2.5,-.7) }*+{\bullet_{v_{10}}}="j"
!{(0,0) }*+{\bullet_{v_{j+3}}}="l" 
!{(2,-.65) }*+{\cdot}
!{(1.2,-.45) }*+{\cdot}
!{(.4,-.2) }*+{\cdot}
!{(1.6,-.55) }*+{\cdot}
!{(.8,-.35) }*+{\cdot} 
"a"-"d"
"a"-"e"
"a"-"f"
"a"-"g"
"a"-"h"
"a"-"i"
"a"-"j"
"a"-"l"
"b"-"d"
"b"-"e"
"b"-"f"
"b"-"g"
"b"-"h"
"b"-"i"
"b"-"j"
"b"-"l"
}
$
\end{center}

Here we have $G = K_{2, j+1}$ after the first two moves which isolates a vertex leaving a $K_{2,j}$.  

\end{Figure}

By our inductive assumption, the second player will win the $K_{2,j}$. Hence $P_2$ wins the $K_{2,j}$ for all $j \geq 1$ and $\Delta$ on a vertex in the partition of size 2.
\end{proof}

Now consider the $K_{2,j} + e_{12}$ with $\Delta$ still on a vertex in the partite set of size 2 and the same vertex enumeration as above.

\begin{Figure}

\begin{center}
$
\xygraph{
!{(2.5,2) }*+{\Delta^{v_1}}="a"
!{(5.88,2) }*+{\bullet^{v_2}}="b"
!{(8.4,0) }*+{\bullet_{v_3}}="c"
!{(7.56,-.5) }*+{\bullet_{v_4}}="d"
!{(6.72,-.7) }*+{\bullet_{v_5}}="e"
!{(5.88,-.85) }*+{\bullet_{v_6}}="f"
!{(5,-1) }*+{\bullet_{v_7}}="g"
!{(4.2,-1) }*+{\bullet_{v_8}}="h"
!{(3.36,-.85) }*+{\bullet_{v_9}}="i"
!{(2.5,-.7) }*+{\bullet_{v_{10}}}="j"
!{(0,0) }*+{\bullet_{v_{j+2}}}="l" 
!{(2,-.65) }*+{\cdot}
!{(1.2,-.45) }*+{\cdot}
!{(.4,-.2) }*+{\cdot}
!{(1.6,-.55) }*+{\cdot}
!{(.8,-.35) }*+{\cdot}
"a"-"b"
"a"-"c"
"a"-"d"
"a"-"e"
"a"-"f"
"a"-"g"
"a"-"h"
"a"-"i"
"a"-"j"
"a"-"l"
"b"-"c"
"b"-"d"
"b"-"e"
"b"-"f"
"b"-"g"
"b"-"h"
"b"-"i"
"b"-"j"
"b"-"l"
}
$
\end{center}

\end{Figure}

We will call this the $SSB_j$ graph of order $j$ (Figure 3.3).  When the order of the graph is understood or insignificant, we will simply write $SSB$.  Removing $e_{12}$ on the first move yields a $K_{2,j}$ with $\Delta$ on $v_2$.  This lends itself to the following corollary:

\begin{Corollary}
The first player will win the $SSB_j$ for any $j$ when $\omega(e) = 1$ for all $e \in E(SSB_j)$ and $\Delta$ is on $v_1$ or $v_2$.
\end{Corollary}

\begin{proof} The first player removes $e_{12}$ and lets $P_2$ start on the $K_{2,j}$ with $\Delta$ on a vertex in the partite set of size two, guaranteeing $P_1$ the win by the previous theorem.  
\end{proof}

It is not the case that $P_1$ will always win the $SSB$ if $\omega(e) \neq 1$ for every edge.  The winner can be determined by similar arguments as those for even cycles.  
\end{section}

\begin{section}{The Complete Graph}

In Corollary  2.4, the first player has no option but to move back to either $v_1$ or $v_2$ since all other vertices only have degree 2.  Suppose now that $P_1$ had more options so that the move is not forced back to $v_1$ or $v_2$ in the $SSB$.  We continue to assume $\omega(e) = 1$ but give $P_1$ more options by adding edges between the vertices in the partition of size $j$ in the $SSB$.  We show next that additional edges do not affect a player's strategy to play the $SSB$ when such a structure exists as a subgraph.

\begin{Lemma}
Assume that $G = K_n$ and that $\omega(e) = 1$ for all $e \in E(G)$.  Then $P_1$ can force $P_2$ to move within the confines of an $SSB_{n-2}$ contained in $K_n$.
\end{Lemma}

\begin{proof}
Assume $G=K_n$ with $\Delta = v_1$ and $\omega(e) = 1$ for all $e \in G$.  Then all of $P_1$'s moves are identical.  Without loss of generality, assume that $P_1$ moves from $v_1$ to $v_2$.  

Then we have $O(P_2, v_2) = \{v_3, v_4, \ldots, v_n\}$ and each option is identical.  So assume without loss of generality that $P_2$ moves to $v_3$.  With $P_1$ on $\Delta = v_3$ there are two non-isomorphic moves for $P_1$.  One of these is to move to $v_1$ and the other is to move to one of the $v_4, v_5, \ldots, v_n$.  Since we want to show that $P_1$ can move along the $SSB$, he would naturally choose the $v_1$ option.

Now $O(P_2, v_1) = \{v_4, v_5, \ldots, v_n\}$ and all of these moves are identical.   Assume that $P_2$ moves to $v_4$.  Then since $v_2 \in O(P_1, v_4)$ we know $P_1$, in keeping with the strategy to move along the $SSB$, will choose to move to $v_2$.

Continuing on in this manner we will have that $v_1 \notin O(P_2, v_2)$, $v_2 \notin O(P_2, v_1)$ since $e_{12}$ was the first edge removed.  In general, every option at every move is identical for $P_2$.  Since $v_1 \in O(P_1, \Delta = v_i)$ for all $v_i \in O(P_2, v_2)$ and $v_2 \in O(P_1, v_j)$ for all $v_j \in O(P_2, v_1)$, $P_1$ is able to choose to move along the $SSB$.  

Keeping up game play in this fashion, i.e., $P_1$ choosing to move to whichever of the $v_1$ or $v_2$ options exist in $O(P_1, \Delta)$ and $P_2$'s moves identical, we will exhaust the edges incident with $v_1$ and $v_2$ leaving $P_2$ on an isolated vertex.  Precisely, if $n$ is even, $P_2$ will be stuck on $v_2$, and if $n$ is odd, $P_2$ will be stuck on $v_1$.
\end{proof}

Notice that since $P_1$ never opted to use any edges outside of the $SSB$, the existence of those edges did not affect the strategy of $P_1$.  We will call the technique of $P_1$ continually choosing to move to $v_1$ or $v_2$ from $\Delta$ the \textbf{SSB strategy} and employ this strategy in Theorem 4.4 below.

\begin{Definition}
We say two distinct vertices are \textbf{mutually adjacent} if they have the same set of neighbors and are neighbors themselves.
\end{Definition}

\begin{Definition}
If two adjacent vertices of degree $k+1$ have $k$ common neighbors, we will call them \textbf{k-mutually adjacent}.
\end{Definition}

Thus saying a graph contains two $k$-mutually adjacent vertices implies that the graph contains an $SSB$ subgraph of order $k$.  We will also speak of vertices that are $k$-mutually adjacent without being adjacent to each other.  Notice that this implies the graph contains a $K_{2,k}$ subgraph.

\begin{Theorem}
Let $G$ be a graph with $\omega(e) = 1$ for all $e \in E(G)$.  If there exists at least two mutually adjacent vertices in $G$ with $\Delta$ at one such vertex, then $P_1$ will win $G$.
\end{Theorem}

\begin{proof}
Assume that $G$ is a graph of order $n$ with $\omega(e) = 1$ for all $e \in E(G)$.  Assume further that $v_1$ and $v_2$ are mutually adjacent.  We proceed by induction on the $k$-mutual adjacency.

If $v_1$ and $v_2$ are 1-mutually adjacent and $\Delta = v_1$ then $d(v_1) = d(v_2) = 2$ and both are adjacent to some other vertex, say $v_3$.  When $P_1$ moves to $v_2$ we have $O(P_2, v_2) = \{v_3\}$ forcing $P_2$'s move.  Then $P_1$ moves to $v_1$ for the win.  Notice that this is consistent with the $SSB$ strategy.

Assume that for all $k \leq j$ the first player to move on a graph $G$ with at least two $k$-mutually adjacent vertices $v_1$ and $v_2$ and $\Delta \in \{v_1, v_2\}$ wins $G$ by moving from $v_1$ to $v_2$ on the first move and continually choosing the $v_1$ or $v_2$ option.  This implies that the second player to move from $G-e_{12}$ and $\Delta \in \{v_1, v_2\}$ wins by employing the same strategy which we are calling the $SSB$ strategy.

Assume $G$ is a graph of order $n$ with $(j+1)$-mutually adjacent vertices $v_1$ and $v_2$ for $1 < j < n-2$ and $\Delta = v_1$.  Enumerate the vertices of $G$ in such a way that $O(P_1,v_1) = \{v_{2}, \ldots, v_{j+3}\}$.  Suppose that $P_1$ moves to $v_2$.  Then $O(P_2, v_2) = \{v_3, \ldots, v_{j+3}\}$.  Without loss of generality, assume that $P_2$ moves to $v_3$.  Since $v_1 \in O(P_1,v_3)$,  let $P_1$ move to $v_1$.  Now $O(P_2, v_1) = \{v_4, \ldots, v_{j+3}\}$.  Thus we have $P_2$ on a $j$-mutually adjacent graph minus $e_{12}$.  This means $P_2$ is on a complete bipartite subgraph of order $j$ contained in $G$.  By Theorem 3.1, the second player to start from a bipartite graph will win, and by Lemma 4.1, since $P_1$ can force $P_2$ to move within the confines of this structure, $P_1$ will win this graph.  Thus $P_1$ wins every graph $G$ with at least two $(j+1)$-mutually adjacent vertices and $\Delta$ on a mutually adjacent vertex. 
\end{proof}

\begin{Corollary}
Assume that $G=K_n$ and that $\omega(e) = 1$ for all $e \in K_n$.  Then $P_1$ can win the $K_n$ for all $n>1$.
\end{Corollary}

\begin{proof}
When $n = 2$ or 3 we have graphs that have been reduced to trivial wins for $P_1$.  Any two vertices in the $K_n$ are $(n-2)$-mutually adjacent.  Thus for $\Delta$ at any vertex, $P_1$ will win the complete graph.
\end{proof}

We have now successfully solved the problem of complete graphs when each edge has weight one.  As shown, the existence of the $SSB$ structure and appropriate starting position solves a large class of graphs.  A quick check will show that the $SSB$ strategy will not work for the complete graph and arbitrary weight assignments.  However, we can show that for $n \leq 7$ the first player can win the complete graph with any weight assignment.  To do this, we modify the $SSB$ strategy slightly to account for the additional options given to the second player.

\end{section}

\vspace{.4in}

\begin{section}{Acknowledgments}
Thank you very much to my advisor, Dr. Warren Shreve, for all of his encouragement in writing this paper, as well as all of the time he spent helping me through the writing process.  Many thanks also to Dr. Sean Sather-Wagstaff, Dr. Joshua Lambert, and Christopher Spicer for all of their insightful comments on the paper.

\end{section}

\bibliographystyle{plain}

\end{document}